\documentclass[11pt]{amsart}
\usepackage{a4wide}
\usepackage{amssymb}
\usepackage{color}
\usepackage[all]{xy}

\newcommand{\Ra}{\Rightarrow}

\newcommand{\w}{\omega}

\newcommand{\supp}{\mathrm{supp}}
\newcommand{\IZ}{\mathbb Z}
\newcommand{\IN}{\mathbb N}

\newtheorem{theorem}{Theorem}

\newtheorem{lemma}{Lemma}

\newtheorem{problem}{Problem}
\theoremstyle{definition}
\newtheorem{example}{Example}
\newtheorem{definition}{Definition}

\title[Separability, precompactness and narrowness in topological groups]{Generalizing separability, precompactness and narrowness in topological groups}
\author{Taras Banakh, Igor Guran, Alex Ravsky}
\address{T.Banakh: Ivan Franko National University of Lviv (Ukraine), and Jan Kochanowski University in Kielce (Poland)}
\email{t.o.banakh@gmail.com}
\address{I.Guran: Ivan Franko National University of Lviv (Ukraine)}
\email{igor-guran@ukr.net}
\address{O.Ravsky: Pidstryhach Institute for Applied Problems of Mechanics and Mathematics National Academy of Sciences of Ukraine}
\email{alexander.ravsky@uni-wuerzburg.de}

\begin{document}
\begin{abstract} We discuss various modifications of separability, precompactness and narrowness in topological groups and test those modifications in the permutation groups $S(X)$ and $S_{<\w}(X)$.
 \end{abstract}
\maketitle

In this paper we define and analyze various properties related to separability and narrowness in topological groups and test these properties for the permutation groups $S(X)$ and $S_{<\w}(X)$. All topological groups in this paper are Hausdorff.

For a set $X$ by $S(X)$ we denote the permutation group of $X$, and by $S_{<\w}(X)$ the normal subgroup of $X$, consisting of permutations $f:X\to X$ having finite support $\supp(f)=\{x\in X:f(x)\ne x\}$. The groups $S(X)$ and $S_{<\w}(X)$ carry the topology of pointwise convergence, i.e., the topology inherited from the Tychonoff power $X^X$ of the discrete space $X$.

For subsets $A,B$ of a group $G$ let $AB=\{ab:a\in A,\;b\in B\}$ be their pointwise product in $G$. For a topological group $(G,\tau)$ by $\tau_*$ we denote the family of open neighborhoods of the identity $1_G$ in $G$.

 By $\w$ and $\w_1$ we denote the smallest infinite and uncountable cardinals, respectively. For a set $X$ and a cardinal $\kappa$, let  $[X]^{<\kappa}=\{A\subseteq X:|A|<\kappa\}$. Therefore, $[X]^{<\w}$ and $[X]^{<\w_1}$ are the families of finite and countable subsets of $X$, respectively.

As a motivation of subsequent definitions, let us consider characterizations of separable and precompact topological groups.

A topological space $X$ is {\em separable} if it contains a countable dense subset of $X$. Separable topological groups admit the following (trivial) characterization.

\begin{theorem}\label{t:sep} For any topological group $G$ the following conditions are equivalent:
\begin{enumerate}
\item $G$ is separable;
\item $\exists S_1\in[G]^{<\w_1}\;\forall U_2\in\tau_*$ such that $S_1U_2=G$;
\item $\exists S_1\in[G]^{<\w_1}\;\forall U_2\in\tau_*$ such that $U_2S_1=G$;
\item $\exists S_1\in[G]^{<\w_1}\;\forall U_2\in\tau_*$ such that $U_2S_1U_2=G$.
\end{enumerate}
\end{theorem}

Following \cite{BGR}, we define a topological group $G$ to be {\em duoseparable} if $G$ contains a countable subset $S$ such that $SUS=G$ for every neighborhood $U\subseteq G$ of the unit. By \cite{BGR}, every topological group is a subgroup of a duoseparable topological group, which means that the duoseparability is a strictly weaker property than the separability.

Next, we discuss (Roelcke) precompact and (Roelcke) narrow topological groups.
A topological group $G$ is {\em precompact} (resp. {\em narrow}) if for any  neighborhood $U$ of the neutral element $1_G$ there is a finite (resp. countable) subset $S\subseteq G$ such that $SU=G=US$.
 It is well-known \cite[3.7.17]{AT} that a topological group $G$ is precompact if and only if $G$ is a subgroup of a compact topological group. By  \cite{Gur} (see also \cite[3.4.23]{AT}), a topological group is narrow if and only if it is topologically isomorphic to a subgroup of the Tychonoff product of second-countable topological groups.

Precompact topological groups admit the following characterization (see \cite[15.81]{Us} or \cite[4.3]{BT}), which resembles the characterization of separability in Theorem~\ref{t:sep}.

\begin{theorem}\label{t:precomp} For any topological group $G$ the following conditions are equivalent:
\begin{enumerate}
\item $G$ is precompact;
\item $\forall U_1\in \tau_*\;\exists S_2\in [G]^{<\w}$ such that $U_1S_2=G$;
\item $\forall U_1\in \tau_*\;\exists S_2\in [G]^{<\w}$ such that $S_2U_1=G$;
\item $\forall U_1\in\tau_*\;\exists S_2\in [G]^{<\w}$ such that $S_2U_1S_2=G$.
\end{enumerate}
\end{theorem}

A topological group $G$ is called {\em Roelcke precompact} (resp. {\em Roelcke narrow}) if for any  neighborhood $U$ of the neutral element $1_G$ there is a finite (resp. countable) subset $S\subseteq G$ such that $USU=G$. By \cite{Us}, every topological group is a subgroup of a Roelcke precompact group, which implies that the Roelcke precompactness is a strictly weaker property than the precompactness and the  Roelcke narrowness is strictly weaker than the narrowness.

Theorems~\ref{t:sep} and \ref{t:precomp} motivate the following definitions that fit into a general scheme. We start with properties that generalize (Roelcke) precompactness and (Roelcke) narrowness.

\begin{definition}\label{d:n} Let $\kappa,\lambda$ be infinite cardinals. A topological group $G$ is called
\begin{itemize}
\item $\mathsf u_1\mathsf s_2^{\kappa}$ if $\forall U_1\in \tau_*\;\;\exists S_2\in [G]^{<\kappa}$ such that $U_1S_2=G$;
\item $\mathsf s_2^{\kappa}\mathsf u_1$ if $\forall U_1\in \tau_*\;\;\exists S_2\in [G]^{<\kappa}$ such that $S_2U_1=G$;
\item $\mathsf s_2^{\kappa}\mathsf u_1\mathsf s_2^{\kappa}$ if $\forall U_1\in\tau_*\;\;\exists S_2\in [G]^{<\kappa}$ such that $S_2U_1S_2=G$;
\item $\mathsf u_1\mathsf s_2^{\kappa}\mathsf u_1$ if $\forall U_1\in \tau_*\;\;\exists S_2\in [G]^{<\kappa}$ such that $U_1S_2U_1=G$;
\item $\mathsf u_1\mathsf s_2^{\kappa}\mathsf u_1\mathsf s_3^\lambda\mathsf u_1$ if $\forall U_1\in \tau_*\;\exists S_2\in [G]^{<\kappa}\;\exists S_3\in[G]^{<\lambda}$ such that $U_1S_2U_1S_3U_1=G$;
\item $\mathsf u_1\mathsf s_2^{\kappa}\mathsf u_3$ if $\forall U_1\in \tau_*\;\exists S_2\in [G]^{<\kappa}\;\forall U_3\in\tau_*$ such that $U_1S_2U_3=G$;
\item $\mathsf u_1\mathsf s_2^{\kappa}\mathsf u_3\mathsf s_4^{\lambda}$ if $\forall U_1\in \tau_*\;\exists S_2\in [G]^{<\kappa}\;\forall U_3\in\tau_*\;\exists S_4\in[G]^{<\lambda}$ such that $U_1S_2U_3S_4=G$;
\item $\mathsf s^\kappa_2\mathsf u_1\mathsf s_3^{\lambda}\mathsf u_4$ if $\forall U_1\in \tau_*\;\exists S_2\in [G]^{<\kappa}\;\exists S_3\in[G]^{<\lambda}\;\forall U_4\in\tau_*$ such that $S_2U_1S_3U_4=G$.
\end{itemize}
\end{definition}
Observe that the properties (2),(3),(4) of Theorem~\ref{t:precomp} coincide with the properties $\mathsf u_1\mathsf s_2^{\w}$, $\mathsf s_2^{\w}\mathsf u_1$, $\mathsf s_2^{\w}\mathsf u_1\mathsf s_2^{\w}$,  respectively. The narrowness is equivalent to the conditions $\mathsf u_1\mathsf s_2^{\w_1}$ and $\mathsf s_2^{\w_1}\mathsf u_1$, but is strictly stronger than $\mathsf s_2^{\w_1}\mathsf u_1\mathsf s_2^{\w_1}$, see Theorem~\ref{t:main}.

Next, we introduce some generalizations of the (duo)separability.

\begin{definition}\label{d:s} Let $\kappa,\lambda$ be infinite cardinals. A topological group $G$ is called
\begin{itemize}
\item $\mathsf s_1^{\kappa}\mathsf u_2$ if $\exists S_1\in [G]^{<\kappa}\;\;\forall U_2\in \tau_*$ such that $S_1U_2=G$;
\item $\mathsf u_2\mathsf s_1^{\kappa}$ if $\exists S_1\in [G]^{<\kappa}\;\;\forall U_2\in \tau_*$ such that $U_2S_1=G$;
\item $\mathsf s_1^{\kappa}\mathsf u_2\mathsf s_1^{\kappa}$ if $\exists S_1\in [G]^{<\kappa}\;\;\forall U_2\in \tau_*$ such that $S_1U_2S_1=G$;
\item $\mathsf s_1^{\kappa}\mathsf u_2\mathsf s_1^{\kappa}\mathsf u_2$ if $\exists S_1\in [G]^{<\kappa}\;\;\forall U_2\in \tau_*$ such that $S_1U_2S_1U_2=G$;
\item $\mathsf u_2\mathsf s_1^{\kappa}\mathsf u_2$ if $\exists S_1\in [G]^{<\kappa}\;\;\forall U_2\in \tau_*$  such that $U_2S_1U_2=G$;
\item $\mathsf s_1^{\kappa}\mathsf u_2\mathsf s_3^\lambda$ if $\exists S_1\in [G]^{<\kappa}\;\;\forall U_2\in \tau_*\;\;\exists S_3\in[G]^{<\lambda}$ such that $S_1U_2S_3=G$;
\item $\mathsf s_1^{\kappa}\mathsf u_2\mathsf s_3^\lambda\mathsf u_4$ if $\exists S_1\in [G]^{<\kappa}\;\;\forall U_2\in \tau_*\;\;\exists S_3\in[G]^{<\lambda}\;\;\forall U_4\in \tau_*$ such that $S_1U_2S_3U_4=G$.
\end{itemize}
\end{definition}

Observe that the conditions (2), (3), (4) of Theorem~\ref{t:sep} coincide with the properties $\mathsf s_1^{\w_1}\mathsf u_2$, $\mathsf u_2\mathsf s_1^{\w_1}$, $\mathsf u_2\mathsf s_1^{\w_1}\mathsf u_2$, respectively. The duoseparability coincides with $\mathsf s_1^{\w_1}\mathsf u_2\mathsf s_1^{\w_1}$.

Following the general scheme we could introduce infinitely many properties extending those in Definitions~\ref{d:n} and \ref{d:s}. But we wrote down only the properties that will appear in Theorem~\ref{t:main} and Example~\ref{ex1}, which are the main results of this paper.

For any topological group we have the following implications (for the unique nontrivial implication $\mathsf s_2^{\w_1}\mathsf u_1\mathsf s_3^\w\Ra\mathsf u_1\mathsf s_2^{\w_1}$, see Lemma 3.31 in \cite{Pachl}).
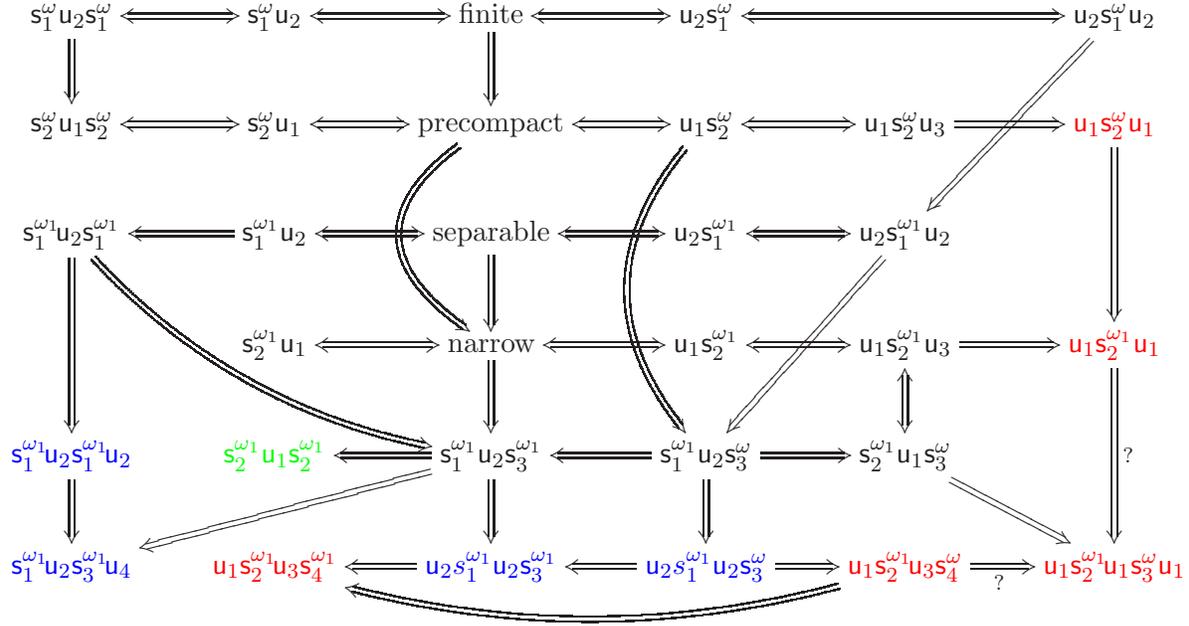
$$
\xymatrix{
\mathsf s_1^\w \mathsf u_2 \mathsf s_1^\w\ar@{<=>}[r]\ar@{=>}[d]&\mathsf s_1^\w \mathsf u_2\ar@{<=>}[r]&\mbox{finite}\ar@{=>}[d]\ar@{<=>}[r]&\mathsf u_2\mathsf s_1^\w\ar@{<=>}[rr]&&\mathsf u_2\mathsf s_1^{\w}\mathsf u_2\ar@{=>}[ddl]\\
\mathsf s_2^\w \mathsf u_1\mathsf s_2^\w\ar@{<=>}[r]&\mathsf s_2^\w \mathsf u_1\ar@{<=>}[r]&
\mbox{precompact}\ar@{<=>}[r]\ar@{=>}@/_35pt/[dd]&\mathsf u_1\mathsf s_2^\w\ar@{<=>}[r]\ar@{=>}@/_30pt/[ddd]&\mathsf u_1\mathsf s_2^{\w}\mathsf u_3\ar@{=>}[r]&{\color{red}\mathsf u_1\mathsf s_2^{\w}\mathsf u_1}\ar@{=>}[dd]
\\
\mathsf s_1^{\w_1}\!\mathsf u_2\mathsf s_1^{\w_1}\ar@{=>}[dd]\ar@{=>}@/_18pt/[rrdd]&\mathsf s_1^{\w_1} \mathsf u_2 \ar@{=>}[l]&
\mbox{separable}\ar@{=>}[d]\ar@{<=>}[r]\ar@{<=>}[l]&\mathsf u_2\mathsf s_1^{\w_1}\ar@{<=>}[r]&\mathsf u_2\mathsf s_1^{\w_1}\mathsf u_2\ar@{=>}[ddl]&
\\
&\mathsf s_2^{\w_1}\mathsf u_1\ar@{<=>}[r]&
\mbox{narrow}\ar@{<=>}[r]\ar@{=>}[d]&\mathsf u_1\mathsf s_2^{\w_1}\ar@{<=>}[r]&\mathsf u_1\mathsf s_2^{\w_1}\mathsf u_3\ar@{<=>}[d]\ar@{=>}[r]&{\color{red}\mathsf u_1\mathsf s_2^{\w_1}\mathsf u_1}\ar@{=>}^{?}[dd]
\\
{\color{blue}\mathsf s_1^{\w_1}\!\mathsf u_2\mathsf s_1^{\w_1}\!\mathsf u_2}\ar@{=>}[d]&{\color{green}\mathsf s_2^{\w_1} \mathsf u_1\mathsf s_2^{\w_1}}&\mathsf s_1^{\w_1}\mathsf u_2\mathsf s_3^{\w_1}\ar@{=>}[l]\ar@{=>}[dll]\ar@{=>}[d]&\mathsf s_1^{\w_1}\mathsf u_2\mathsf s_3^\w\ar@{=>}[d]\ar@{=>}[r]\ar@{=>}[l] &\mathsf s_2^{\w_1}\mathsf u_1\mathsf s_3^\w\ar@{=>}[dr]&
\\
{\color{blue}\mathsf s_1^{\w_1}\!\mathsf u_2\mathsf s_3^{\w_1}\!\mathsf u_4}&{\color{red}\mathsf u_1\mathsf s_2^{\w_1}\!\mathsf u_3\mathsf s_4^{\w_1}}&{\color{blue}\mathsf u_2s_1^{\w_1}\mathsf u_2\mathsf s_3^{\w_1}}\ar@{=>}[l]&{\color{blue}\mathsf u_2s_1^{\w_1}\mathsf u_2\mathsf s_3^\w}\ar@{=>}[r]\ar@{=>}[l]&{\color{red}\mathsf u_1\mathsf s_2^{\w_1}\!\mathsf u_3\mathsf s_4^{\w}}\ar@{=>}@/^20pt/[lll]\ar@{=>}_(.45){?}[r]&{\color{red}\mathsf u_1\mathsf s_2^{\w_1}\!\mathsf u_1\mathsf s_3^{\w}\mathsf u_1}\\
}
$$
\vskip20pt

The implications which are not labeled by the question mark cannot be reversed, which is either obvious or witnessed by the examples constructed in Theorem~\ref{t:main} and Examples~\ref{ex0}, \ref{ex1}. A property in the diagram is drawn with
\begin{itemize}
\item the {\color{red}{\em red}} color if it holds for both groups $S_{<\w}(\w_1)$ and $S(\w_1)$;
\item the {\color{green}{\em green}} color if it holds for $S_{<\w}(\w_1)$ but not for $S(\w_1)$;
\item the {\color{blue}{\em blue}} color if it holds for $S(\w_1)$ but not for $S_{<\w}(\w_1)$;
\item the {\em black} color if it does not hold neither for $S_{<\w}(\w_1)$ nor for $S(\w_1)$.
\end{itemize}

\begin{theorem}\label{t:main} For a set $X$ of infinite cardinality $\kappa$, the topological group
\begin{enumerate}
\item $S_{<\w}(X)$ is $\mathsf u_1\mathsf s_2^\w\mathsf u_1$, $\mathsf s_2^{\w_1}\mathsf u_1\mathsf s_2^{\w_1}$, $\mathsf u_1\mathsf s_2^{\w_1}\mathsf u_3\mathsf s_4^\w$, but is neither  $\mathsf s_1^\kappa\mathsf u_2\mathsf s_3^\kappa\mathsf u_4$ nor $\mathsf u_2\mathsf s_1^\kappa\mathsf u_2\mathsf s_3^\kappa$;
\item $S(X)$ is $\mathsf u_1\mathsf s_2^\w\mathsf u_1$, $\mathsf s_1^{\w_1}\!\mathsf u_2\mathsf s_1^{\w_1}\mathsf u_2$, $\mathsf u_2\mathsf s_1^{\w_1}\!\mathsf u_2\mathsf s_3^\w$, but not $\mathsf s_2^{\kappa}\mathsf u_1\mathsf s_2^\kappa$.
\end{enumerate}
\end{theorem}

\begin{proof} The proof of this theorem will be divided into a series of lemmas.

\begin{lemma} The topological groups  $S_{<\w}(X)$ and $S(X)$ are $\mathsf u_1\mathsf s_2^\w\mathsf u_1$.
\end{lemma}

\begin{proof} Let $G$ denote the group $S(X)$ or $S_{<\w}(X)$.

Given any neighborhood $U_1\in\tau_*$ of the identity $1_G$, find a finite subset $A\subset X$ such that $U_1\supseteq V=\{f\in G:\forall a\in A\;\;f(a)=a\}$. Let $B\subset X\setminus A$ be any set with $|B|=|A|$.
Let $S_2$ be any finite subset of $G$ such that for each injective function $g:A\to A\cup B$ there exists a permutation  $f\in S_2$ such that $f{\restriction} A=g$. We claim that $U_1S_2U_1=G$. Given any $h\in G$, find a permutation $v\in V$ such that  $v(h(A)\setminus A)\subseteq B$.  Then $v\circ h(A)=v(h(A)\cap A)\cup v(h(A)\setminus A)\subseteq A\cup B$. The choice of $S_2$ yields a permutation $s\in S_2$ such that $s{\restriction}A=v\circ h{\restriction}A$. It follows that the permutation $u=s^{-1}\circ v\circ h$ belongs to the set $V$ and hence
$h=v^{-1}\circ s\circ u\in U_1S_2U_1$, witnessing that $G$ is $\mathsf u_1\mathsf s_2^\w\mathsf u_1$.
\end{proof}

\begin{lemma} The topological group $G=S_{<\w}(X)$ is $\mathsf s_2^{\w_1}\mathsf u_1\mathsf s_2^{\w_1}$.
\end{lemma}

\begin{proof} Given any neighborhood $U_1\in\tau_*$ of the identity of $G$, find a finite set $A\subset X$ such that the subgroup $\{f\in G:\forall a\in A\;\;f(a)=a\}$ is contained in $U_1$. Choose a countable infinite set $S_2\subseteq G$ such that the family $(s(A))_{a\in S_2}$
consists of pairwise disjoint sets and $s=s^{-1}$ for every $s\in S_2$. We claim that $G=\bigcup_{s\in S_2}sU_1s^{-1}$. Indeed, given any permutation $h\in G$ find $s\in S_2$ such that $s(A)$ is disjoint with the finite set $\supp(h)$. Then the permutation $u=s^{-1}\circ h\circ s$ belongs to the set $U_1$ and hence $h\in sU_1s^{-1}=sU_1s\subseteq S_2U_1S_2$, which means that $G$ is $\mathsf s_2^{\w_1}\mathsf u_1\mathsf s_2^{\w_1}$.
\end{proof}

\begin{lemma}\label{l:usus} The topological group $G=S_{<\w}(X)$ is $\mathsf u_1\mathsf s_2^{\w_1}\mathsf u_3\mathsf s_4^\w$.
\end{lemma}

\begin{proof} Given any neighborhood $U_1\in\tau_*$ of the identity in $G$, find a finite set $A_1\subset X$ such that $U_1\supset V_1:=\{f\in G:\forall a\in A_1\;\;f(a)=a\}$. Let $\mathcal A_1$
be an arbitrary countably infinite family of mutually disjoint
subsets of $X$ such that $|A|=|A_1|$ for each $A\in\mathcal A_1$.
Let $S_2=S_2^{-1}\subset G$ be a countable set such that for each
$A'_1\in\mathcal A_1$ and each bijection $f_1:A_1\to A'_1$ there exists a permutation $f_2\in S_2$ such that
$f_2{\restriction}A_1=f_1$. Next, let $U_3\in \tau_*$ be any neighborhood of the identity in $G=S_{<\w}(X)$. Then there exists a finite subset $A_3\subset X$ such that $U_3\supseteq V_3:=\{f\in S_{<\w}(X):\forall a\in A_3\;\;f(a)=a\}$.
Let $\mathcal A_3$ be a finite family of pairwise disjoint
subsets of $X$ such that $|\mathcal A_3|>|A_1|$ and $|A|=|A_3|$ for each $A\in\mathcal A_3$. Finally choose a finite subset $S_4=S_4^{-1}$ of $S_{<\w}(X)$ such that for each
$A'_3\in\mathcal A_3$ there exists $f_4\in S_4$ such that $f_4(A_3)=A'_3$.
We claim that $S_{<\w}(X)=U_1S_2U_3S_4$. Indeed, let $h$ be any element of $G$.
Since $|\mathcal A_3|>|A_1|$, there exists a set $A'_3\in\mathcal A_3$ such that
$A'_3\cap h^{-1}(A_1)=\varnothing$.
Pick $f_4\in S_4$ such that $f_4(A_3)=A'_3$. Then $f_4^{-1}h^{-1}(A_1)\cap A_3=\varnothing$.
Choose a set $A'_1\in\mathcal A_1$ such that $A'_1\cap A_3=\varnothing$.
Pick an arbitrary $v_3\in V_3$ such that $v_3f_4^{-1}h^{-1}(A_1)=A'_1$. There exists $f_2\in S_2$ such that $f_2(a)=v_3f_4^{-1}h^{-1}(a)$ for each $a\in A_1$.
Then $f_2^{-1}v_3f_4^{-1}h^{-1}(a)=a$ for each $a\in A_1$ and hence the permutation $\mathsf u_1=f_2^{-1}v_3f_4^{-1}h^{-1}$ belongs to $V_1\subseteq U_1$.
Then $h=u_1^{-1}f_2^{-1}v_3f_4^{-1}\in U_1S_2U_3S_4$ and hence $S_{<\w}(X)$ is
$\mathsf u_1\mathsf s_2^{\w_1}\mathsf u_3\mathsf s_4^\w$.
\end{proof}

\begin{lemma} The topological group $G=S(X)$ is $\mathsf u_2\mathsf s_1^{\w_1}\mathsf u_2\mathsf s_3^\w$.
\end{lemma}

\begin{proof} Fix a permutation $h\in S(X)$ such that for every $x\in X$ the points $h^n(x)$, $n\in\IZ$, are pairwise distinct. Let $S_1=\{h^n\}_{n\in\IZ}$. Given any neighborhood $U_2\subseteq G$ of the identity, find a finite set $A\subset X$ such that $U_2\supseteq\{g\in G:\forall a\in A\;\;g(a)=a\}$. Choose a finite family $S_3\in[G]^{<\w}$ such that $|S_3|=|A|+1$ and the family $(s(A))_{s\in S_3}$
consists of pairwise disjoint sets and $s=s^{-1}$ for every $s\in S_3$.

We claim that $U_2S_1U_2S_3=G$. Given any permutation $g\in G$, find $s_3\in S_3$ such that $s_3(A)\cap g^{-1}(A)=\emptyset$. Such permutation $s_3$ exists since the family $(s(A))_{s\in S_3}$
consists of $|g^{-1}(A)|+1$ many pairwise disjoint sets. Then $gs_3(A)\cap A=\emptyset$. The choice of the permutation $h$ ensures that $s_1(A)\cap A=\emptyset$ for some $s_1\in S_1$.

Since the set $gs_3(A)\cup s_1(A)$ is disjoint with the set $A$, we can find a permutation $u_2\in U_2$ such that $u_2s_1{\restriction}A=gs_3{\restriction}A$.
Then the permutation $u_2'=s_1^{-1}u_2^{-1}gs_3$ belongs to the neighborhood $U_2$ and hence $g=u_2s_1u_2's_3^{-1}=u_2s_1u_2's_3\in U_2S_1U_2S_3$.
\end{proof}






\begin{lemma} The topological group $G=S_{<\w}(X)$ is not  $\mathsf s_1^\kappa\mathsf u_2\mathsf s_3^\kappa\mathsf
u_4$, where $\kappa=|X|$.
\end{lemma}

\begin{proof} We should prove that $\forall S_1\in[G]^{<\kappa}\;\exists U_2\in\tau_*\;\forall S_3\in [G]^{<\kappa}\;\exists U_4\in\tau_*$ such that $S_1U_2S_3U_4\ne G$.

Given any $S_1\in[G]^{<\kappa}$ find a point $a\in X\setminus\bigcup_{s\in S_1}\supp(s)$ and consider the open neighborhood $U_2=\{f\in G:f(a)=a\}$. For any set $S_3\in[G]^{<\kappa}$, we can find a point $b\in X\setminus\{f^{-1}(a):f\in S_3\}$ and consider the neighborhood $U_4=\{f\in G:f(b)=b\}\in\tau_*$.

We claim that for any $f\in S_1U_2S_3U_4$, we have $f(b)\ne a$.
Find $s_1\in S_1,u_2\in U_2,s_3\in S_3,u_4\in U_4$ such that $f=s_1u_2s_3u_4$. The choice of $a\notin\supp(s_1)$ ensures that $s_1u_2(a)=s_1(a)=a$ and the choise of $b$ guarantees that $b\ne s_3^{-1}(a)$ and hence $s_3(b)\ne a$. Then  $f(b)=s_1u_2s_3(b)\ne a$ and $S_1U_2S_3U_4\ne G$.
\end{proof}

\begin{lemma}  The topological group $G=S_{<\w}(X)$ is not $\mathsf u_2\mathsf s_1^\kappa\mathsf u_2\mathsf
s_3^\kappa$, where $\kappa=|X|$.
\end{lemma}

\begin{proof} Assuming that $G$ is $\mathsf u_2\mathsf s_1^\kappa\mathsf u_2\mathsf s_3^\kappa$, find a set $S_1\in[G]^{<\kappa}$ such that for any neighborhood $U_2\subseteq G$ of the identity there exists a set $S_3\in[G]^{<\kappa}$ such that $U_2S_1U_2S_3=G$. Choose any point $a\in X\setminus\bigcup_{s\in S_1}\supp(s)$ and consider the neighborhood $U_2=\{g\in G:g(a)=a\}$ of the identity in $G$. The choice of $S_1$ yields a set $S_3\in [G]^{<\kappa}$ such that $G=U_2S_1U_2S_3$. Choose any permutation $g\in G$ such that $g^{-1}(a)\notin \{s^{-1}(a):s\in S_3\}$. Since $g\in G=U_2S_2U_2S_3$, there are permutations $u_2,u_2'\in U_2$, $s_1\in S_1$, $s_3\in S_3$ such that $g=u_2s_1u_2's_3$.  Then for the point $b=s_3^{-1}(a)$, we get $g(b)=u_2s_1u_2's_3(b)=u_2s_1u_2'(a)=a$ and hence $g^{-1}(a)=b\in\{s^{-1}(a):s\in S_3\}$, which contradicts the choice of $g$.
\end{proof}

\begin{lemma} The topological group  $S(X)$ is $\mathsf s_1^{\w_1}\mathsf u_2\mathsf s_1^{\w_1}\mathsf u_2$.
\end{lemma}

\begin{proof} Choose a permutation $h\in S(X)$ such that for every $x\in X$ the points $h^n(x)$,
$n\in\IZ$, are pairwise distinct. Consider the countable subset $S_1=\{h^n:n\in\IZ\}$ of the group $S(X)$. We claim that $S_{<\w}(X)\subseteq S_1U_2S_1$
for every neighborhood $U_2\in\tau_*$ of the identity in $S(X)$.

Given any $U_2\in\tau_*$, find a finite set $A\subset X$ such that the subgroup $V_2=\{f\in S(X):\forall a\in A\; f(a)=a\}$ is contained in $U_2$. The choice of the permutation $h$ guarantees that for any
permutation $f\in S_{<\w}(X)$, there is $n\in\IZ$ such that $h^n(A)\cap\supp(f)=\emptyset$.
This implies $h^{-n}\circ f\circ h^n\in V_2$ and hence $f\in h^nV_2h^{-n}\subseteq S_1U_2S_1$. The density of the subgroup $S_{<\w}(X)\subseteq S_1U_2S_1$ in $S(X)$ implies that $S_1U_2S_1U_2=S(X)$.
\end{proof}

\begin{lemma} The topological group $G=S(X)$ is not $\mathsf s_2^\kappa\mathsf u_1\mathsf s_2^\kappa$.
\end{lemma}

\begin{proof}  Fix any point $a\in X$ and consider the neighborhood $U_1=\{f\in G:f(a)=a\}$. Assuming that $G$ is $\mathsf s_2^\kappa\mathsf u_1\mathsf s_2^\kappa$, we can find a set $S_2\in[G]^{<\kappa}$ such that $G=S_2U_1S_2$. Choose any permutation  $g\in S(X)$ such that
$g\big(\{t^{-1}(a):t\in S_2\})\cap \{s(a):s\in S_2\}=\varnothing$. Find permutatuons $s,t\in S_2$ and $u\in U_1$ such that $g=sut$ and observe that $gt^{-1}(a)=su(a)=s(a)$, which contradicts the choice of $g$. This contradiction shows that
$G$ is  not $\mathsf s_2^\kappa\mathsf u_1\mathsf s_2^\kappa$.
\end{proof}
\end{proof}

In the following examples by $\IZ$ we denote the additive group of integers, endowed with the discrete topology.

\begin{example}\label{ex0} In the Tychonoff power $\IZ^{\w_1}$ of  $\IZ$, consider the dense subgroup
$$G=\big\{(x_i)_{i\in\w_1}\in\IZ^{\w_1}:|\{i\in\w_1:x_i\ne 0\}|<\w\big\}$$
and observe that $G$ is narrow but not $\mathsf s_1^{\w_1}\mathsf u_2\mathsf s_1^{\w_1}\mathsf u_2$ and not $\mathsf s_1^{\w_1}\mathsf u_2\mathsf s_3^\w$.
\end{example}

\begin{example}\label{ex1} Let $X$ be a nonseparable Banach space and $X\rtimes \IZ$ be the product $X\times \IZ$ endowed with the group operation $*$ defined by
$$(x,n)*(y,m)=(x+2^ny,n+m).$$ The topological group $X\rtimes\IZ$ is  $\mathsf s_1^{\w_1}\mathsf u_2\mathsf s_1^{\w_1}$ but not $\mathsf u_1\mathsf s_2^{\w_1}\mathsf u_1\mathsf s_3^{\w}\mathsf u_1$.
\end{example}

\begin{proof} Identify $X$ and $\IZ$ with the subgroups $X\times\{0\}$ and $\{0\}\times\IZ$ of the semidirect product $G=X\rtimes\IZ$. It is clear that the binary operation $*:G\times G\to G$ is continuous and so is the inversion $(\cdot)^{-1}:G\to G$, $(\cdot)^{-1}:(x,n)\mapsto (-2^{-n}x,-n)$. Therefore, $(G,*)$ is a topological group. It can be shown (see \cite{BGR}) that $G=\IZ*U*\IZ$ for any neigborhood $U\subseteq X$ of zero, which means that $G$ is  $\mathsf s_1^{\w_1}\mathsf u_2\mathsf s_1^{\w_1}$.

To see that $G$ is not $\mathsf u_1\mathsf s_2^{\w_1}\mathsf u_1\mathsf s_3^{\w}\mathsf u_1$, let $U_1$ be the open unit ball $\{x\in X:\|x\|<1\}$ of the Banach space $X$. Assuming that $G$ is $\mathsf u_1\mathsf s_2^{\w_1}\mathsf u_1\mathsf s_3^{\w}\mathsf u_1$, we can find sets $S_2\in[G]^{<\w_1}$ and $S_3\in[G]^{<\w}$ such that $U_1S_2U_1S_3U_1=G$. Find a separable Banach subspace $H$ of $X$ such that $S_2S_3\subseteq H\times\IZ$. Since the Banach space $X$ is not separable, $X\ne H$. By the Hahn-Banach Theorem, there exists a linear continuous functional $f:X\to\mathbb R$ such that $f(H)=\{0\}$ and $\|f\|=1$. Find $m\in\IN$ such that $S_3\subset X\times[-m,m]$ and choose a point $x\in X$ with $|f(x)|>2+2^m$. Find elements $u_1,u_1',u_1''\in U_1$, $s_2\in S_2$, $s_3\in S_3$ such that $x=u_1s_2u'_1s_3u_1''$. Then $x=u_1s_2s_3s_3^{-1}u_1's_3u_1''$. Write the elements $s_2s_3\in H\times \IZ$ and $s_3\in X\times[-m,m]$ as $s_2s_3=(h_2,n_2)$, $s_3=(h_3,n_3)$ for some $h_2,h_3\in X$ and $n_2,n_3\in\IZ$ with $h_2\in H$ and $n_3\in[-m,m]$.
Then $$
\begin{aligned}
(x,0)=\;&(u_1,0)*(h_2,n_2)*(h_3,n_3)^{-1}*(u_1',0)*(h_3,n_3)*(u_1'',0)=\\
&(u_1+h_2,n_2)*(-2^{-n_3}h_3,-n_3)*(u_1',0)*(h_3,n_3)*(u_1'',0)=\\
&(u_1+h_2,n_2)*(2^{-n_3}u_1',0)*(u_1'',0)=(u_1+h_2+2^{n_2-n_3}u_1'+2^{n_2}u_1'',n_2)
\end{aligned}
$$
and hence $n_2=0$ and $x=u_1+h_2+2^{-n_3}u_1'+u_1''$. 
Then $$|f(x)|\le |f(u_1)|+|f(h_2)|+2^{-n_3}|f(u_1')|+|u_1''|\le 1+0+2^{-n_3}\cdot 1+1\le 2+2^m,$$
which contradicts the choice of $x$. This contradiction shows that $G$ is not $\mathsf u_1\mathsf s_2^{\w_1}\mathsf u_1\mathsf s_3^{\w}\mathsf u_1$.
\end{proof}

We finish this paper by a problem suggested by implications with question marks in the diagram.

\begin{problem} Is there a  $\mathsf u_1\mathsf s_2^{\w_1}\!\mathsf u_1\mathsf s_3^{\w}\mathsf u_1$ topological group which is neither $\mathsf u_1\mathsf s_2^{\w_1}\mathsf u_1$ nor $\mathsf u_1\mathsf s_2^{\w_1}\!\mathsf u_3\mathsf s_4^{\w}$?
\end{problem}

\section*{Acknowledgement}

The authors express their sincere thanks to Jan Pachl for his valuable comment on the implication  $\mathsf s_2^{\w_1}\mathsf u_1\mathsf s_3^\w\Ra\mathsf u_1\mathsf s_2^{\w_1}$ (proved in Lemma 3.31 of his book \cite{Pachl}).

\end{document}